\documentclass[english, 11pt]{amsart} %Xtop
\usepackage{graphics}
\usepackage{cite}
\usepackage{amsmath}
\usepackage{amsfonts}
\usepackage{amssymb}
\usepackage{float}
\usepackage{wrapfig}
\usepackage{epic}
\usepackage{eepic}
\usepackage{ae}
\usepackage{mathrsfs,array,graphicx}
\usepackage[all,ps]{xy}
\usepackage{hyperref}
\usepackage{epstopdf}
\usepackage{note}
\usepackage{a4wide}
\usepackage{std_macros}
\newcommand{\der}{\textup{der}}

\title{Equidistribution in supersingular Hecke orbits}

\author{Arno Kret}
\begin{document}

\begin{abstract}
We prove that Hecke operators act with equidistribution on the basic stratum of certain Shimura varieties. We relate the rate of convergence to the bounds from the Ramanujan conjecture of certain cuspidal automorphic representations on $\Gl_n$ for which this conjecture is known, and therefore we obtain optimal estimates on the rate of convergence. 
\end{abstract}

\maketitle

\section*{Introduction}

\nocite{chaitemp}

Let $(G, X)$ be a Shimura datum. In particular, $G$ is a reductive group over $\Q$, and $X$ is a Hermitian symmetric domain. Let $\Gamma \subset G(\R)$ be a torsion-free congruence subgroup and let $S$ be the quotient $\Gamma \backslash X$. Let $\{ T_m \}$ ($m \in \N$) be a sequence of Hecke operators of $G$, then the operators $T_m$ act on the space $S$. Equidistribution means roughly that the images of points $x \in S$ spread out equally over $S$ under the action of the sequence of operators $\{T_m\}$. 

Let us make this more precise. The space $S$ is of finite volume with respect to its Haar measure $\dd \mu$. The group $G(\R)$ acts by multiplications on the right on $S$, and by translations on the space $L^2(S)$ of square integrable complex valued functions on $S$. This way $L^2(S)$ is a representation of $G(\R)$, and carries an action of the Hecke operators. Now, \emph{equidistribution} is true for the quadruple $(G, X, \Gamma, \{T_m\})$, if, for all $f \in L^2(S)$, the pointwise limit 
\begin{equation}\label{intro_AA}
\lim_{m \to \infty} \frac {T_m(f)}{\deg(T_m)}, 
\end{equation}
converges to the constant function on $S$ whose value is equal to the integral
\begin{equation}\label{intro_BB}
\int_S f \dd \mu. 
\end{equation}

A typical example is the case where $G = \GL_2$, $X$ is the complex upper half plane, and $S$ a modular curve, say $S = Y_1(N)$ where $N \in \N$ with $N \geq 4$. On $Y_1(N)$ we have the classical operators $T_m$ from the theory of modular forms~\cite{diamond.modforms}. In~\cite[\S2, p.~197]{MR1827734} is shown that equidistribution is true for $(\Gl_2, X, \Gamma_1(N), \{T_m\})$. 

When equidistribution is established, the next step is to give bounds on the rate of convergence of the limit in Eq.~\eqref{intro_AA}. Optimal bounds require precise knowledge on the shape of the discrete spectrum of reductive groups (Arthur's conjectures) and the Ramanujan conjecture for cuspidal automorphic representations of $\GL_n$. Unfortunately, these conjectures are currently proved only in certain restricted ---~but significant~--- cases. Thus, most known bounds %on the rate of convergence come from known approximations to the Ramanujan conjecture.
are only approximations of the expected rate of convergence.% to what one expects in general.

More generally, equidistribution has been studied for other quotients of real reductive groups, which are not Shimura varieties. For example, one might consider the group $G = \GL_n$ for $n \geq 3$ and put $X := \PGL_n(\Z) \backslash \PGL_n(\R)$. On the group $\Gl_n$ one has Hecke operators $T_{r, m}$ analogous to the $T_m$ for $\Gl_2$. For a definition see Equation~\eqref{thefunction} in the text although there we write $T_{r,m}^+$ for them and we work with the closely related group $\GL_1 \times \GL_n$. In~\cite[Thm 1.2]{MR1827734} Clozel and Ullmo show that equidistribution is true in this setting for $\GL_n$ and the operators $T_{r, m}$. In~\cite{MR2058609} quotients $\Gamma \backslash G(\R)$ of more general classes of real reductive groups are considered. 

In this article we take equidistribution to a different direction: We study equidistribution in characteristic $p$. Let us return from the more and more general quotients $\Gamma \backslash G(\R)$ to the case where $S$ is a Shimura variety. Then, apart from being a complex analytical object, $S$ has the natural structure of a smooth quasi-projective variety (Baily-Borel) and is canonically defined over a certain number field $E$ called the reflex field of $S$ (Shimura)~\cite{MR0498581}. To simplify, let us restrict henceforward to the case where $S$ is a Shimura variety of PEL type; this means that $S$ parametrizes Abelian varieties with additional PEL structures (Polarization, Endomorphisms, Level structures). Let $\p$ be a prime of $E$ where $S$ has \emph{good reduction}; this means in particular that the PEL type moduli problem extends to a problem over $\cO_{E_\p}$, and that this extended problem is representable by a smooth and quasi-projective scheme $\cS/\cO_{E_\p}$ \cite[\S5]{MR1124982}. Write $p$ for the rational prime number below $\p$ and write $S_\p := \cS \otimes \fq$, where $\fq$ is the residue field of $\p$. One may wonder if equidistribution also holds in the mod $p$ variety $S_\p$ in a suitable sense (to give meaning to ``in a suitable sense'' is part of the question). %. --- certainly the measures involved in Equations~\eqref{intro_AA}, \eqref{intro_BB} defined in the complex setting do not help in this new modulo $p$ setting.
The \emph{initial idea} is to find a definition of a measure $\mu$ on the set $S_\p(\lfq)$, and prove that for complex valued functions $f$ on $S_\p(\lfq)$, the limit in Equation~\eqref{intro_AA} converges to an integral with respect to $\mu$. However, this idea is far too naive: The special fibre $S_\p$ has several stratifications, and any correct statement of equidistribution has to take these stratifications into account. For us the most important stratification on $S_\p$ is the \emph{Newton stratification}~\cite[\S7]{MR2141705}. By definition, the Newton strata of $S_\p$ are  the loci in $S_\p$ where the isogeny class of the $p$-divisible group $\cA_x[p^\infty]$ with additional PEL-structures is constant. Here, $\cA_x$ is the Abelian variety with PEL structures corresponding to the point $x \in S_\p(\lfq)$. The Newton stratification of $S_\p$ is a canonical decomposition of $S_\p$ into a finite union of locally closed subvarieties, and these subvarieties are stable under the action of the Hecke correspondences. Consequently, any point lying in one of those strata will stay inside this stratum under the action of all Hecke operators. Thus one does not expect the points modulo $p$ to spread out equally over $S_\p$, whatever `spread out' should mean here. In particular, the initial idea we sketched above is false. 

One could refine the equidistribution question, and ask: ``Is there then equidistribution inside the individual Newton strata?'' The answer to this question is ``no'' again in general, because the Newton strata can often be refined further, and these further refinements are also stable under the Hecke operators (outside $p$). However, we show in this article that there \emph{are} a large number of interesting cases where equidistribution is true inside the \emph{basic} Newton stratum. We consider essentially\footnote{We also treat cases where the unitary group $U \subset G$ arises from restriction of scalars of a totally real extension $F^+$ over $\Q$. Then, in particular $U$ splits over a certain CM extension $F$ of $\Q$. See Section~\ref{sect.extension} for the complete list of conditions.} the case where the group $G_\R$ is of the form $GU(n-1, 1)$, and $\p$ is a `split' place, which means that $G_{\qp}$ is isomorphic to $\Gm \times \Res_{M/\qp} \Gl_{n, M}$ for some finite \'etale $\qp$-algebra $M$. The corresponding Shimura varieties $S$ are certain arithmetic quotients of the complex unit ball. In the case at hand the basic stratum $B \subset S_\p$ is finite \'etale over $\fq$, and we prove equidistribution with respect to the \emph{counting measure} on the finite set $B(\lfq)$. 

Let us, before giving the precise statement of our results, start with an example. Consider again the modular curve $Y_1(N)$ ($N \geq 4$) and let $p$ be prime number that is coprime to $N$ and let $B$ be the basic stratum of $Y_1(N)_{\Fp}$. By definition $B(\lfp)$ is the set of pairs $(E, \eta) \in Y_1(N)(\lfp)$ of elliptic curves $E$ equipped with a point $\eta$ of order $N$ where $E$ is \emph{supersingular}. The famous Eichler-Deuring mass formula (cf.~\cite[Thm.~12.4.5]{MR772569}) says that 
$$
\# B(\lfq) = \deg(\Gamma_1(N)) \frac {p-1}{24}, 
$$ 
where $\deg(\Gamma_1(N))$ is the degree of the $\Gamma_1(N)$ moduli problem. 

Apart from studying the cardinal of $B(\lfq)$, one may also study how the Hecke correspondences act on it. In particular the sequence of Hecke operators $\{T_m\}$ considered above in the complex setting, also acts with equidistribution on the supersingular locus $B(\lfq)$: Let $A$ be the free complex vector space with basis $B(\lfq)$. Equip $A$ with a (any) norm $|\cdot |$, and define for $f \in A$, $\Avg(f) \in A$ to be the constant function on $B(\lfq)$ whose value is $\frac 1{\# B(\lfq)} \sum_{x \in B(\lfq)} f(x)$. There exists a constant $C \in \R_{>0}$ such that for any $\eps > 0$ there exists an integer $M$ such that for all $m > M$, coprime to $N p$, we have \begin{equation}\label{menaresresult} \left| \frac {T_m(f)}{\deg(T_m)} - \Avg(f)  \right| < C \cdot m^{\eps - \frac 1 2}. \end{equation} In particular the sequence $T_m(f)/\deg(T_m)$ converges to the average function $\Avg(f)$. The result in Equation~\eqref{menaresresult} is due to Menares~\cite{menares}. He stated the corresponding result for the basic stratum of the algebraic stack $Y_1(1)$ of elliptic curves with no additional structures, but the proof is the same. In fact, we learned the idea of studying equidistribution in supersingular Hecke orbits of more general Shimura varieties from Menares's article. 

Let us now give the statement of our results. We consider a Shimura variety of PEL-type, of type A, as considered by Kottwitz in~\cite{MR1124982}. We assume fixed a PEL-datum consisting of: A simple algebra $T$ over a CM field $F$; a positive involution on the algebra $T$ inducing complex conjugation on $F$; a Hermitian $T$-module $(V, \langle \cdot, \cdot \rangle)$, where $\langle \cdot, \cdot \rangle$ is symplectic; a morphism $h \colon \C \to \End_T(V)_\R$ of $\R$-algebras such that $h(\li z) = h(z)^*$ for all $z \in \C$, and the symmetric bilinear form $V_\R \otimes V_\R \to \R$, $(x, y)\mapsto (x, h(i) y)$ is positive definite. Let $G/\Q$ be the algebraic group such that for every $\Q$-algebra $R$ the group $G(R)$ consists of the elements $g \in \GL(V \otimes R)$ respecting the pairing $\langle \cdot, \cdot \rangle$ up to a scalar. We restrict $h$ to $\C^\times \subset \C$, and write $X$ for the $G(\R)$-conjugacy class of the morphism $h^{-1}$; then $(G, X)$ is a Shimura datum. We assume that the CM field $F$ is of the form $\cK F^+$, with $\cK/\Q$ a quadratic imaginary extension of $\Q$ and $F$ a totally real extension of $\Q$. We let $p$ be a prime of good reduction in the sense of Kottwitz~\cite[\S5]{MR1124982} and we assume that $p$ splits in $\cK/\Q$. Let $K \subset G(\Af)$ be a compact open subgroup of the form $K = K_p K^p$, with $K_p \subset G(\qp)$ hyperspecial and $K^p \subset G(\Afp)$ sufficiently small so that the PEL-type moduli problem of level $K$ is defined over $\cO_E \otimes \Z_{(p)}$ and the variety $S = S_K$ representing that moduli problem is smooth and quasi-projective. Pick an $E$-prime $\p$ above $p$ and consider the reduced variety $S_\p$. Let us give the definition of the basic, also called supersingular, Newton stratum of $S_\p$. Recall that the Newton strata of $S_\p$ are the loci in $S_\p$ where the isogeny class of the $p$-divisible group $\cA_x[p^\infty]$ with additional PEL-structures is constant. The isogeny class of $\cA_x[p^\infty]$ is determined by its rational Dieudonn\'e module ${\mathbb D}(\cA_x[p^\infty])_\Q$, and this Dieudonn\'e module is determined by its slope morphism $\nu$ \cite{MR809866}. If the centralizer of the slope morphism $\nu$ is $G$, then the isocrystal is called basic. Among the isocrystals ${\mathbb D}(\cA_x[p^\infty])_\Q$ for $x \in S(\lfq)$ there is one \emph{unique} basic isocrystal, and the locus in $S_\p$ corresponding to this isocrystal is the basic stratum $B \subset S_\p$. We work in this article under the condition that $B$ is \emph{finite}. This translates to an explicit condition on the signatures of the unitary group of similitudes $G_\R$; the only allowed signatures are $0$ and $1$, and the prime $p$ has to decompose in a certain way in the field $F^+$ (see Proposition~\ref{finitenesshyp} for the precise statement). In future work we hope to say more about the non-finite cases. 

Write $A$ for the free complex vector space on the set $B(\lfq)$. In the text we define an endomorphism $\Avg$ of $A$, in the same spirit as above for the modular curve, however its definition is slightly more involved because it has to take into account contribution from the center of $G$ and the different components of the Shimura variety (see below Equation~\eqref{glob}). Let $T_m$ be a sequence of Hecke operators on $G(\Af)$, satisfying two mild conditions (H1), (H2) (see Section~\ref{section:HeckeOperators}). Each Hecke operator $T_m$ defines a Hecke operator $\Psi(T_m)$ on the cocenter of the group $G$ (see Equation~\ref{psimapping}). To each operator $T_m$ we attach a certain norm $N(T_m)$ which is the sum of the absolute values of the coefficients appearing in the Satake transform of $T_m$ (see Section~\ref{section:HeckeOperators}). We prove:

\begin{theorem}
Let $v \in A$ be an element. There exists a constant $C \in \R_{>0}$ such that for any $\eps > 0$ there exist a natural number $M$ such that for all $m > M$ we have
$$
\left| \frac {T_m(v)}{\deg(T_m)} - \Psi(T_m)(\Avg(v)) \right| \leq C \cdot \left| \frac {N(T_m)}{\deg(T_m)}\right|. 
$$
\end{theorem}

This means that if the sequence $T_m$ is such that $\lim_{m \to \infty} \frac {N(T_m)}{\deg(T_m)} = 0$, then the sequence $T_m$ acts with equidistribution (this is called condition (H3) in the text below). We expect this to be the optimal result, in the sense that, if the condition (H3) is not satisfied, then equidistribution is false for the sequence $T_m$.

In Section~\ref{sect:example} we give a typical example of sequence $T_m$ to which the above theorem applies. In the beginning of the introduction we mentioned the sequence of Hecke operators $T_{r, m}$ for $\GL_n$ for which Clozel and Ullmo prove equidistribution in the classical setting. Over the extension $\cK/\Q$ our group $G$ becomes isomorphic to a general linear group (times a $\Gm$-factor) and we can define $T_{r,m}$ (for $G$) as the Hecke operator that is obtained from $T_{r,m}$ on $G_\cK$ via base change from $\cK$ to $\Q$ (see Equation~\eqref{theheckeoperator}). We then have

\begin{corollary}
Let $v \in A$ be an element. Then there exists a constant $C \in \R_{>0}$ such that for any $\eps > 0$ there exist an index $M$, such that for all square-free integers $m > M$ and all $r$ with $1 \leq r \leq n-1$ we have
$$
\left| \frac {T_{r, m}(v)}{\deg(T_{r, m})} - \Psi(T_{r, m})(\Avg(v)) \right| \leq C m^{\eps -[F:\Q] \tfrac {r(n-r)}2}. 
$$
\end{corollary}

How does one prove such analytical statements about varieties in characteristic $p$? The starting point is the formula of Kottwitz for the number of points on a Shimura variety of PEL type over a finite field (see the introduction of \cite{MR1124982}). Originally this formula is used to connect the cohomology of Shimura varieties with the theory of automorphic forms. Nowadays, that problem may be understood as the one of describing the $\ell$-adic cohomology of Shimura varieties as Hecke/Galois modules. In our work we restrict the formula of Kottwitz so that it counts only the contribution of the basic stratum. The resulting restricted formula can still be stabilized as in~\cite{MR1044820}, and the resulting stable formula can be compared with the geometric side of the trace formula. This yields a description of the $\ell$-adic cohomology of the basic stratum in terms of automorphic representations of the group $G$, or in endoscopic situations, automorphic representations of the endoscopic groups of $G$. In the article~\cite{kret1} we carried out this program for certain simple Shimura varieties, the ``Kottwitz varieties''. In the cases where the basic stratum $B \subset S_\p$ is finite we get a description of the free complex vector space on the set $B(\lfq)$ as Galois/Hecke module, in terms of automorphic representations of the group $G$ (we give the statement in Equation~\eqref{previoustheorem}). The automorphic description of the Hecke module $\textup{Map}(B(\lfq), \C)$ then gives the necessary analytical input to prove equidistribution.

%To end this introduction, we make a short speculation on the case where $B$ is not finite. When the basic stratum is non-finite it contains finer sub-stratifications that are stable by the outside-$p$-Hecke operators. For example, one has the superspecial locus inside the basic stratum. This locus is \emph{finite}, see the remark below Theorem 2 in the introduction of~\cite{viehmann}. One does not expect equidistribution in the whole basic stratum, but only inside these finer strata. Perhaps one expects it to hold in the superspecial locus? If one wants to use methods similar to those of this paper to prove equidistribution in the finer strata, one first has to show that the cohomology of finer strata admits an automorphic description for the outside-$p$-Hecke operators. Currently it is not even known if one should expect such a description to exist. We hope to be able to return these questions in future work.

\ 

{

\center{*}

}

\ 

Let us now discuss the structure of this article. The article is roughly cut into two parts. Part 1 spans Sections 1 to 4, in Section 5 we give an example, and Part 2 spans from Sections 6 to 8. In Part 1 we show that equidistribution is true in the basic stratum in case the Shimura variety $S$ is assumed on top of the above conditions, to be a \emph{Kottwitz variety} as considered in~\cite[\S1]{MR1163241}. The Kottwitz varieties $S$ are considerably more simple, for instance these varieties are projective and issues with endoscopy do not play a role. This greatly simplifies the trace formula for the group $G$. Section~4 contains the proof for equidistribution; we use our automorphic description of the cohomology of $B$ from~\cite[Thm~3.13]{kret1} to prove equidistribution.  In Part 2 we attack the problem of equidistribution for varieties which need not be projective and the group has endoscopy. Before we can establish equidistribution in this case, we need to work more. First we need to extend some results of the article~\cite{kret1} to the current, larger classes of endoscopic Shimura varieties. In Section~\ref{sect.extension} we extend the automorphic description of the cohomology of the basic stratum from~[\textit{loc. cit.}, Thm.~3.13] to a suitable statement valid for a substantially larger class of non-compact endoscopic varieties. We require, roughly, that of one the signatures of the unitary group is coprime with $n$ (Hypothesis~\ref{hyp.noendoc}). We verify that this condition is always satisfied in case the basic stratum is finite. The assumption forces the endoscopic contribution to the cohomology of the basic stratum to vanish (Lemma~\ref{prop.noendoscopy}). Such results could be of independent interest (see also the statements of Theorem~\ref{section7thm} and Proposition~\ref{finitenesshyp}). Finally, in Section~\ref{sect.final} we deduce the equidistribution statement for the endoscopic varieties.

\bigskip

\noindent \textbf{Acknowledgements:} I thank my thesis adviser Laurent Clozel for pointing out that Thm 3.13 from~\cite{kret1} will allow one to prove equidistribution. I also thank my second advisor Laurent Fargues, Ariyan Javanpeykar and finally I also thank Emmanuel Ullmo who indicated the article of Menares.

\setcounter{tocdepth}{1}
\tableofcontents

\section{Some simple Shimura varieties}\label{sect.first}

In the first part of this article we focus on the Shimura varieties as considered by Kottwitz~\cite{MR1163241}. These varieties are associated to a division algebra $D$ whose center is a CM field $F$. The algebra $D$ is equipped with an involution $*$ of the second kind, \ie $*$ induces complex conjugation on $F$. We embed $F$ into the complex numbers, and we assume the field $F$ splits into a compositum $F = \cK F^+$ of a quadratic imaginary number field $\cK \subset \C$ and a totally real number field $F^+$. For any commutative $\Q$-algebra $R$, the group $G(R)$ is by definition the group of elements $g \in (D \otimes_\Q R)^\times$ such that $gg^* \in R^\times$. The couple $(G, X)$ is then a Shimura datum of PEL type. If $K \subset G(\Af)$ is a compact open subgroup, sufficiently small, then we have a variety $S = S_K$ defined over the reflex field $E$, the variety $S$ represents a moduli problem of PEL type, see~\cite[\S5]{MR1124982}. Let $\p$ be an $E$-prime where the variety $S$ has good reduction in the sense defined by Kottwitz~\cite[\S5]{MR1124982}. In particular $S$ extends to a smooth projective $\cO_{E_\p}$-scheme. We write $\Fq$ for the residue field of $E$ at $\p$, and we put $S_\p := S \otimes \Fq$. Let $p$ be the rational prime number under $\p$. We fix an embedding $\nu_p \colon E \to \lqp$ compatible with $\p$. Throughout this article we assume that the prime number $p$ is split in the field $\cK$. 

\newcommand{\Auniv}{\cA^{\textup{univ}}}
\renewcommand{\D}{{\mathbb D}}

%Let $B$ be the basic stratum of $S_\p$ (see
We now define the basic locus $B$ of $S_\p$. %\cite{MR2141705}. %We will be slightly short, for more details we refer to, e.g.~\cite{MR2141705}. 
We introduce the following notations:
\begin{itemize}
\item write $\Auniv$ for the universal Abelian variety over $S$ and $\lambda, i, \li \eta$ for its additional PEL type structures~\cite[\S6]{MR1124982};
\item $L$ is the completion of the maximal unramified extension of $\qp$ contained in $\lqp$;
\item for $\alpha > 0$  positive integer, write $E_{\p, \alpha} \subset \lql$ for the unramified extension of degree $\alpha$ of $E_\p$, and write $\fqa$ for the residue field of $E_{\p, \alpha}$; 
\item $\sigma$ is the arithmetic Frobenius of $L/\qp$;
\item $V$ is the $D^\textup{opp}$-module with space $D$ where an element $d \in D^{\opp}$ acts on the left through multiplication on the right on the space $D$.
\end{itemize}
Let $x \in S(\lfq)$ be a point. The couple $(\lambda, i)$ defined on the fibre $\Auniv_x$ induces via the functor $\D(\square)\otimes L$ additional structures on the isocrystal $\D(\Auniv_x)_L$. There exists an isomorphism $\varphi \colon V \otimes L \isomto \D(\cA_{K,x})_L$ of skew-Hermitian $B$-modules~\cite[p.~430]{MR1124982}, and via this isomorphism we can send the crystalline Frobenius on $\D(\cA_{K,x})_L$ to a $\sigma$-linear operator acting on $V_L$. This operator on $V \otimes L$ may be written in the form $\delta \cdot (\id_V \otimes \sigma)$ with $\delta \in G(L)$ independent of $\varphi$ up to $\sigma$-conjugacy. Let $B(G_{\qp})$ be the set of all $\sigma$-conjugacy classes in $G(L)$ (cf.~\cite{MR809866}). We have constructed a mapping $\xi_p \colon S(\lfq) \to B(G_{\qp}), x \mapsto \delta$. The fibres of the map $\xi_p$ are the \emph{Newton strata} of $S_\p$. The \emph{basic stratum} is one particular Newton stratum, it is defined as follows. To any $\delta \in G(L)$, one may attach a \emph{slope morphism} $\nu_\delta$ \cite[4.2]{MR809866}. Let $\D$ be the diagonalizable pro-torus over $\qp$ whose character group is $\Q$. Then $\nu_\delta$ is a morphism $\D \to G_{\qp}$. The element $\delta$ is called \emph{basic} if $\nu_\delta$ has image inside the center of $G_{\qp}$. The image of $\xi_p$ in $B(G_{\qp})$ is finite and it contains exactly one basic element $b$ (nowadays, the image ${\textup{Im}}(\xi_p)$ is known by~\cite{viehmann}). The subset $\xi_p^{-1}(b) \subset S(\lfq)$ is closed for the Zariski topology, and one equips it with the reduced subscheme structure to make it into a scheme: This is the \emph{basic stratum} $B \subset S_\p$. The variety $B$ is defined over $\fq$ and projective, but not smooth: In general one expects that it has  many singularities. In this article we will work in the case where $B$ is \emph{assumed} to be finite, in that case $B$ \emph{is} smooth. 

In fact, among the set of all Kottwitz varieties, it is rarely true that $B$ is finite. However, the class of Kottwitz varieties with finite basic stratum is still quite interesting: It contains all the varieties considered by Harris and Taylor to prove the local Langlands conjecture in their book~\cite{MR1876802}. The condition that $B$ is finite corresponds to a condition on the signatures of the unitary group at infinity and the decomposition of the prime number $p$ in the field $F^+$. More explicitly, let $U \subset G$ be the subgroup of elements with trivial factor of similitudes. Then $U(\R)$ is isomorphic to a product of real, standard unitary groups $U(s_v, n - s_v)$, where $v$ ranges over the infinite $F^+$-places. We may, and do, assume that $s_v \leq \tfrac 12 n$. The field $F^+$ is embedded into $\li \Q \subset \C$ and also in $\li \Q_p$, and therefore the group $\Gal(\lqp/\qp)$ acts on the set of infinite $F^+$-places. 

\begin{lemma}
The variety $B/\fq$ is finite if and only if the following two conditions hold:
\begin{itemize}
\item[(\textit{i})] for all $v$ we have $s_v \in \{0, 1\}$;
\item[(\textit{ii})] for each $\Gal(\lqp/\qp)$-orbit $\wp$ of infinite $F^+$-places, we have $s_v = 1$ for at most one $v \in \wp$. 
\end{itemize}
\end{lemma}
\begin{proof}
This follows (for example) from~\cite[\S4.3]{kret2}.
\end{proof}

From this point onwards we assume that the conditions (\textit{i}) and (\textit{ii}) in the Lemma above are satisfied for our Shimura datum $(G, X)$. Via the embeddings $\li \Q \supset F^+ \subset \lqp$, any $F^+$-place $\wp$ above $p$ defines a $\Gal(\lqp/\qp)$-orbit of embeddings $F^+ \hookrightarrow \li \Q$. We define numbers $s_\wp$ by $s_\wp := \sum_{v \in \wp} s_v$ for each $\wp |p$. Due to our assumption that (\textit{i}) and (\textit{ii}) hold we have $s_\wp \in \{0, 1\}$ for each $\wp$.

Let $A$ be the free complex vector space on the set $B(\lfq)$. 
Let $\cH(G(\Af))$ be the convolution algebra of compactly supported complex valued functions on $G(\Af)$, and write $\cH(G(\Af)//K)$ for the algebra of functions that are $K$-invariant under left and right translations. The Hecke algebra $\cH(G(\Af)//K)$ acts on the variety $B$ through correspondences and on the vector space $A$ via endomorphisms. Let $f_\infty$ be a function at infinity whose stable orbital integrals are prescribed by the identities of Kottwitz in \cite[p.~182, Eq.~(7.4)]{MR1044820}; it can be taken to be (essentially) an Euler-Poincar\'e function \cite[p.~657, Lemma~3.2]{MR1163241} (cf. \cite{MR794744}). The function has the following property: Let $\pi_\infty$ be an $(\ig, K_\infty)$-module occurring as the component at infinity of an automorphic representation $\pi$ of $G$. Then the trace of $f_\infty$ against $\pi_\infty$ equals the Euler-Poincar\'e characteristic $\sum_{i=0}^\infty N_\infty (-1)^i \dim \uH^i(\ig, K_\infty; \pi_\infty \otimes \xi)$, where $N_\infty$ is a certain explicit constant (cf. \cite[Lemma 3.2]{MR1163241}). We use the following space of automorphic forms: $\cA(G) = L^2(G(\Q) \backslash G(\A), \omega)$, where $\omega$ is the trivial central character. In the article~\cite[Thm.~3.13]{kret1} we proved that for every Hecke operator $f^p \in \cH(G(\Af^p)//K)$ we have:
\begin{equation}\label{previoustheorem}
\Tr(f^p \otimes \one_{K_p}, A) = \eps \sum_{\pi \subset \cA(G), \pi_p \textup{ Steinberg type}} \Tr(f_\infty f, \pi^p) + \sum_{\pi \subset \cA(G), \ \dim(\pi) = 1} \Tr(f_\infty f, \pi^p), 
\end{equation}
where the sign $\eps$ equals $(-1)^{t(n-1)}$ with $t$ the number of infinite $F^+$-places $v$ such that $p_v = 1$. ``Steinberg type'' means the following (cf.~\cite[Def.~4]{kret1}): Recall that the group $G(\qp)$ is isomorphic to ${\mathbb Q}_p^\times \times \prod_\wp \GL_n(F^+_\wp)$, where the product ranges over the $F^+$-places $\wp$ dividing $p$. Any representation $\pi_p$ of $G(\qp)$ thus breaks up into a tensor product $\pi_p \cong \otimes_{\wp} \pi_\wp$ of representations $\pi_\wp$ of the groups $\Gl_n(F^+_\wp)$. A smooth representation $\pi_p$ of $G(\qp)$ is of \emph{Steinberg type} if, (\textbf{1}) The factor of similitudes $\qp^\times$ of $G(\qp)$ acts through an unramified character on the space of $\pi_p$, and (\textbf{2}) For all $F^+$-places $\wp$ above $p$, the representation $\pi_\wp$ of $\Gl_n(F_\wp^+)$, is a generic unramified representation if $s_\wp = 0$ or (B) the twist of the Steinberg representation by an unramified characters if $s_\wp = 0$.

\begin{remark}
The module $A$ is semi-simple as $\cH(G(\Af^p))$-module because all its irreducible subquotients lie in the discrete spectrum of $G$.
\end{remark}

We will use the result in Equation~\eqref{previoustheorem} to deduce an equidistribution statement of Hecke operators acting on the basic stratum $B(\lfq)$.

\section{Hecke operators}\label{section:HeckeOperators}

Let $\{T_m\}_{m \in \N} \in \cH(G(\Af))$ be a sequence of Hecke operators such that
\begin{itemize}
\item[(\textbf{H1})]  $T_m = \bigotimes_v T_{m, v}$ is an elementary tensor of local Hecke operators. 
\item[(\textbf{H2})]  $T_{m, S} = \one_{K_S}$ for all $m \in \N$, where $S$ is a finite set outside which the data $(G, K)$ is unramified, $K$ decomposes into $K = K_S K^S$, $K_S \subset G(\A_S)$ compact open and $K^S \subset G(\Af^S)$ hyperspecial. 
\end{itemize}
(The first condition is only there for notational convenience.) 

There is an important additional condition on the sequence $\{T_m\}_{m \in \N}$, before we can prove equidistribution. We define $\deg(T_m) := \Tr(T_m, \one)$. We define a number $N(T_m) \in \R_{\geq 0}$, as follows. Let $\ell$ be a prime number coprime to $S$. Let $T_{m, \ell}$ be the $\ell$-th component of $T_m$. Let $T_\ell$ be a maximal split torus in $G_{\ql}$ with rational Weyl group $W_\ell$, and let $\sum_{\nu \in X_*(T_\ell)} c_{\ell, \nu} \cdot [\nu] \in \C[X_*(T_\ell)]^{W_\ell}$ be the Satake transform of $T_m$. We define $N_\ell(T_{m, \ell}) := \sum_{\nu \in X_*(T)} |c_{\ell, \nu}|$, and $N(T_m) := \prod_{\ell \notin S} N_\ell(T_{m, \ell})$. We require on the sequence $T_m$ that
\begin{equation}\label{condition_H3}
\tag{\textbf{H3}}
\lim_{m \to \infty} \frac {N(T_m)}{\deg(T_m)} = 0, 
\end{equation}
(cf. Equation~\eqref{intheproofofthemaintheorem}). In Section~\ref{sect:example} we give an example of a sequence of Hecke operators satisfying conditions (H1), (H2) and (H3).

We define the Hecke algebra $\cT \subset C_c^\infty(G(\Af))$ to be the complex algebra generated by the operators $T_m$. The algebra $\cT$ is commutative.

\section{Hecke orbits}

The Hecke algebra $\cT$ does not act transitively on the basic stratum; there are two innocent obstructions: (1) an obstruction from the cocenter of the group $G$, and (2) the Hasse invariant $\Ker^1(G, \Q)$, which need not be trivial. Since the action of $\cT$ on $B(\lfq)$ is not transitive, the set $B(\lfq)$ breaks up into orbits. In this section we define explicitly certain orbits inside $B(\lfq)$. We will then show in our main theorem that the algebra $\cT$ acts transitively and with equidistribution on these explicitly defined orbits.

The obstructions (1) and (2) are what one expects: For the first obstruction (1): If the image of $K \subset G(\Af)$ in the cocenter $C(\Af)$ is not sufficiently large (and this will always be the case for many $C$, due to the presence of Abelian class groups), then the double coset space $G(\Q) \backslash (X \times G(\Af)/K)$ is not connected. A point in one connected component will be sent by a Hecke operator to another connected component only if this operator is non-trivial on the cocenter. The second obstruction (2) is there because $S(\C)$ is \emph{not} equal to the double coset space $G(\Q) \backslash (X \times G(\Af)/K)$. To explain this, we need a little more notation: We write $\Sh(G, K)$ for the variety associated to the Shimura datum $(G, X)$ (and the group $K$), as defined by Deligne in his Corvallis article~\cite{MR0498581}. Thus, 
$$
\Sh(G, K)(\C) = G(\Q) \backslash (X \times G(\Af)/K). 
$$ 
Then $S(\C)$ is the disjoint union
\begin{equation}\label{decomposition}
S(\C) = \coprod_{G'} \Sh(G', K) %G'(\Q) \backslash (X \times G(\Af)/K), 
\end{equation}
ranging over a set of representatives of the set of isomorphism classes of $\Q$-reductive groups $G'$, such that $G$ and $G'$ are isomorphic locally at all places. The latter index set is equal to the kernel $\Ker^1(G: \Q)$.  
$$
\Ker^1(G: \Q) := \Ker \lhk H^1(\Q, G) \to \prod_v H^1(\Q_v, G) \rhk.
$$
The group $\Ker^1(G:\Q)$ group depends only on the cocenter of $G$, and is trivial in case $n$ is even~\cite[p.~393]{MR1124982}. The Hecke correspondences act on the right hand side via their natural action on the double coset spaces. Thus all points in a Hecke orbit will have the same invariant in $\Ker^1(G:\Q)$. 

Let $d \colon G \surjects C$ be the cocenter of the group $G$. We have the morphism $h$ from Deligne's torus ${\mathbb S} := \Res_{\C/\R} {\mathbb G}_{\textup {m}}$ to $G_\R$. By composing this morphism with $d \otimes \R$ we obtain a morphism $h' \colon {\mathbb S} \to C_\R$. The couple $(C, \{h'\})$ is a zero dimensional Shimura datum. The finite Shimura variety $\Sh(C, \{h'\}, d(K))$ parametrizes the connected components of the original variety~\cite[\S2.6]{MR0498581} , \ie the natural morphism
\begin{equation}\label{deligne}
\pi_0(G(\Q) \backslash (X \times G(\Af)/K)) \to C(\Q) \backslash ( \{h\} \times C(\Af)/d(K)), 
\end{equation}
is an isomorphism. Via this mapping, the action of the Hecke operator $f \in \cH(G(\Af))$ on the left hand side coincides with the action of the operator $\Psi(f)$ in $\cH(C(\Af))$ on the right hand side. Here the map $\Psi \colon \cH(G(\Af)) \to \cH(C(\Af))$ is characterized by the condition that, for all $c \in C(\Af)$, and for all $f \in \cH(G(\Af))$ we have:
\begin{equation}\label{psimapping}
\left[ \Psi f \right](c) = \begin{cases} \int_{G_\der(\Af)} f(gh) \dd \mu(h) & \textup{if } c = \li g \in \textup{Im}(G(\Af) \rightarrow C(\Af)) \cr 
0 & \textup{otherwise,}
\end{cases}
\end{equation} 
where the Haar measure on $G_{\der}(\Af)$ is the one giving the group $K \cap G_{\der}(\Af)$ volume $1$. The mapping in Equation~\eqref{deligne} is $\Aut(\C/E)$-equivariant and descents to an isomorphism of $E$-schemes $\pi_0(\Sh(G:K) ) \isomto \Sh(C:d(K))$. The variety $S$ is an union of $\# \Ker^1(G, \Q)$ copies of the variety $\Sh(G:K)$ \cite[\S6]{MR1124982}. We obtain an $E$-isomorphism
\begin{equation}\label{glob}
\pi_0(S) \isomto \coprod_{\Ker^1(G, \Q)} \Sh(C:d(K)). 
\end{equation}
Both sides are finite \'etale $E$-schemes and the $\Gal(\li \Q/E)$-action is unramified at $\p$. Locally at the prime $\p$ we have a natural model of $S$ over the ring of integers $\cO_{E_\p}$, and we construct a model of the right hand side in the straightforward manner: Take the ring of global sections $W$ of the scheme $\coprod_{\Ker^1(G, \Q)} \Sh(C:d(K))_{E_\p}$. Then $W$ is a $\qp$-algebra; let $W^\circ \subset W$ be the integral closure of $\zp$ in $W$. Then $\spec(W^\circ)$ is our integral model. We write $Y = \spec(W^\circ)$ and view it as a as scheme over $\cO_{E, \p}$. We reduce the map in Equation~\eqref{glob} modulo $\p$ and compose with $B \subset S_\p \surjects \pi_0(S_\p)$ to obtain the map $\psi \colon B \to Y$. Let $\Avg \colon A \to A$ be the mapping taking the `average' of an element $v \in A$ along the fibres of the mapping $\psi \colon B \to Y$. More precisely, $\Avg$ is defined as follows. For each point $y \in Y$ we have the fibre $B_y$ of $\psi$ above $y$. Define $A_y$ to be the free complex vector space on the set $B_y(\lfq)$. Then $A$ is the direct sum of the $A_y$ with $y$ ranging over the set $Y(\lfq)$. For each $y \in Y$ we have the map (of vector spaces): $\Psi_y \colon A_y \to \C$, defined by $\sum_{x \in B_y(\lfq)} a_x \cdot [x] \mapsto \sum_{x \in B_y(\lfq)} a_x$. Write $E_y = \sum_{x \in B_y(\lfq)} [x] \in A_y$. Define the endomorphism $\Avg_y \colon A_y \to A_y$, by $v \mapsto \frac {\Psi_y(v)}{\#B_y(\lfq)} \cdot E_y$. Take the direct sum of $\Avg_v$ over all $y \in Y$ to obtain the endomorphism $\Avg \colon A \to A$. We will prove that any element $v \in A$ converges to its average $\Avg(v)$ under the action of the sequence of Hecke operators $T_m \in \cT$. The complex vector space $A$ is finite dimensional and therefore carries a norm $|\cdot |$, uniquely defined up to equivalence of norms. Using this norm we may give the statement of the main theorem:

\begin{theorem}\label{maintheorem}
Let $T_m$ be a sequence of Hecke operators satisfying conditions \textup{(H1)} and \textup{(H2)}. Let $v \in A$ be an element. There exists a constant $C \in \R_{>0}$ such that for any $\eps > 0$ there exist a natural number $M$ such that for all $m > M$ we have
$$
\left| \frac {T_m(v)}{\deg(T_m)} - \Psi(T_m)(\Avg(v)) \right| \leq C \cdot \left| \frac {N(T_m)}{\deg(T_m)}\right|. 
$$
\end{theorem}
\begin{proof}
In Section~\ref{secbeginproof} we prove Theorem~\ref{maintheorem}. 
\end{proof}

Now assume additionally that the sequence $T_m$ satisfies condition (H3) saying that $\lim_{m \to \infty} \frac {N(T_m)}{\deg(T_m)} = 0$. Then Theorem~\ref{maintheorem} gives an equidistribution statement for the sequence $T_m$ acting on the basic stratum. In Section~\ref{sect:example} we explain what Theorem~\ref{maintheorem} says for a concrete example.

\begin{remark} 
Theorem~\ref{maintheorem} gives a bound on the rate of convergence: The sequence $\frac {T_m(v)}{\deg(T_m)}$ convergences to the sequence $\Psi(T_m)(\Avg(v))$ with at least the speed of convergence of the limit $\lim_{m \to \infty} \frac {N(T_m)}{\deg(T_m)} = 0$. This convergence bound is the optimal bound coming from the Ramanujan conjecture.
\end{remark}

Let us discuss in some more detail the statement of Theorem~\ref{maintheorem}. Strictly speaking, Theorem~\ref{maintheorem} does not say that the sequence $(T_m(v) / \deg(T_ m))$ converges; its says that the terms of this sequence get arbitrarily close to the terms of the sequence $\Psi(T_m)(\Avg(v))$. If one chooses a sequence of operators $T_m$ such that the operators $\Psi(T_m)$ are non-trivial for an infinite amount of times, then the sequence $\Psi(T_m)(\Avg(v))$ does not converge; it will vary over a finite set of values. Thus, in those cases, our theorem actually says that the limit $\lim_{m \to \infty} (T_m(v) / \deg(T_m))$ does not exist at all. This behavior is completely normal and understandable: Already in the classic, complex setting, equidistribution is false (in the strict sense) when the group has non-trivial cocenter. To convince the reader, we consider a Shimura variety for the multiplicative group ${\mathbb G}_{\textup{m}}$: Consider the set $S_N := (\Z/N\Z)^\times$ viewed as a variety over $\Q$ by identifying $S_N$ with the $N$-th primitive roots of unity $\mu_N^\times$. A natural sequence of Hecke operators acting on this variety is the sequence of operators $T_p$ induced by the $p$-th Frobenius element acting on $\mu_N^\times(\li \Q)$ ($p$ not dividing $N$). The operator $T_p$ acts on the complex valued functions $f$ on $S_N$ by multiplication: $T_p \cdot f := (x \mapsto f(px))$. Clearly, the limit $\lim_{p \to \infty} T_p(f)$ does not exist in general, only if one imposes a condition like $p \equiv 1 \mod N$.

Thus, returning to our original Shimura datum and our original sequence $T_m$, one can do two things to get equidistribution in the strict sense:
\begin{itemize}
\item Restrict the sequence of Hecke operators to range only over those $T_m$ acting trivially on the cocenter (\ie $\Psi(T_m)$ acts by multiplication with $\deg(T_m)$), yielding certain congruence conditions on the numbers $m$;
\item Restrict the space of vectors $v \in V$, and consider only those $v$ such that the Hecke algebra acts by multiplication with its degree on $\Avg(v)$. 
\end{itemize}
We choose to keep the slightly more complicated statement of convergence involving several limit points. 

\section{Equidistribution for Kottwitz varieties}\label{secbeginproof}

We prove Theorem~\ref{maintheorem}.

\begin{definition}
The character formula for $A$ in Equation~\eqref{previoustheorem} expresses $A$ as a sum of Hecke modules of the form $(\pi_\ff^p)^{K^p}$. We define $A_0 \subset A$ to be the $\cT$-submodule generated by the submodules $(\pi_\ff^p)^{K^p} \subset A$ for $\pi$ an infinite dimensional automorphic representation of $G(\A)$. 
\end{definition}

The following proposition proves the essential part of Theorem~\ref{maintheorem}:

\begin{proposition}\label{mainprop}
Let $v \in A_0$. Let $T_m \in \cH(G(\Af))$ be a sequence of Hecke operators satisfying conditions \textup{(H1)} and \textup{(H2)}. There exists a constant $C \in \R_{> 0}$ such that for all natural numbers $m$ we have
$$
\left| \frac {T_m(v)}{\deg(T_m)} \right| \leq C \cdot \left| \frac {N(T_m)} {\deg(T_m)} \right|.
$$
\end{proposition}

\begin{proof}
Of course, by multiplying with the degree, it suffices to show that $|T_m(v)| \leq C |N(T_m)|$ (for some constant $C$, not necessarily the same constant as in the Theorem). By our Theorem in Equation~\eqref{previoustheorem} it suffices to prove this inequality for each vector $v \in \pi_\ff^K$ in each automorphic representation $\pi$ contributing to the character formula of $A_0$. Let $\pi$ be one of these cuspidal automorphic representations. We base change $\pi$ to an automorphic representation $BC(\pi)$ of $\cK^\times \times D^\times$, viewed as an algebraic group over $\Q$. Since our group is associated to a division algebra, we are using a very classical form of base change. More precisely, $D$ is a division algebra and therefore the second condition in Theorem A.3.1(b) of the Clozel-Labesse appendix of~\cite{MR1695940} is satisfied (alternatively, see~\cite[\S VI.2]{MR1876802}). Currently results exist for more general unitary groups of similitudes, the article~\cite[Thm.~1.1]{shinnote} proves base change for all the unitary groups of similitudes that we consider in the later sections (those groups are not necessarily anisotropic modulo the center). Next we apply the Jacquet-Langlands correspondence to send $BC(\pi)$ to an automorphic representation $\Pi := JL(BC(\pi))$ of the $\Q$-group $G^+ = \Res_{\cK/\Q} \Gm \times \Res_{F/\Q} \Gl_{n,F}$. The first reference where the Jacquet-Langlands correspondence is proved for division algebras, is in the work of Vigneras~\cite{vigneras}. Harris and Taylor also give a proof for division algebras in their book~\cite[\S VI.1]{MR1876802}. Finally Badulescu treats in the article~\cite[Thm.~4.2(a)]{MR2390289} the general linear groups over division algebras. The automorphic representation $\Pi$ is \emph{discrete}. Both base change and Jacquet-Langlands have local variants at all the prime numbers and the places at infinity. Furthermore global base change (resp. Jacquet-Langlands transfer) is compatible with local base change (resp. Jacquet-Langlands transfer). At the prime $p$ we have $G^+(\qp) \cong G(\qp) \times G(\qp)$ and, by local-global compatibility, $\Pi_p$ is isomorphic to $\pi_p \otimes \pi_p$. The representation $\pi_p$ is essentially square integrable because it is an unramified twist of the Steinberg representation, and therefore $\Pi_p$ also has this property. The representation $\Pi$ is then forced to be cuspidal by the classification of Moeglin-Waldspurger of the discrete spectrum~\cite{MR1026752}. Because $\Pi$ is cuspidal the Ramanujan conjecture applies to it. This conjecture is true for $\Pi$ because $\Pi$ is obtained by base change and Jacquet-Langlands from an automorphic representation $\pi$ of an unitary group (of similitudes), implying that $\Pi$ is conjugate self-dual. Furthermore, $\Pi$ is cohomological because $\Tr(f_\infty, \pi_\infty) \neq 0$. For such representations $\Pi$ the Ramanujan conjecture is proved to be true in the articles~\cite{clozelpurity, MR2800722, caran}. Consequently the components $\Pi_\lambda$ are \emph{tempered} $\Gl_n(F_\lambda)$-representations for all primes $\lambda$ of $F$. Let $\ell$ be a prime not lying in $S$. Then $\Pi_\ell$ is the base change of $\pi_\ell$. We know that $\Pi_\ell$ is unramified and tempered, and we know that $\pi_\ell$ is unramified as well. Therefore, $\pi_\ell$ must be a tempered representation of $G(\Q_\ell)$. The representation $\pi_\ell$ is tempered, and therefore the absolute values of the eigenvalues of its Hecke matrix $\varphi_{\pi_\ell}$ are all equal to $1$. Let $\sum_{\nu \in X_*(T_\ell)} c_\nu [\nu] \in \C[T_\ell]^{W_\ell}$ be the Satake transform of $T_{m, \ell}$. Because $\varphi_{\pi_\ell}$ is tempered we have
\begin{equation}\label{intheproofofthemaintheorem}
|\Tr(T_{m, \ell}, \pi_\ell)| = \left|\sum_{\nu \in X_*(T_\ell)} c_\nu [\nu](\varphi)  \right| \leq \sum_{\nu \in X_*(T_\ell)} |c_\nu| = N_\ell(T_m). 
\end{equation}
Take the product over all $\ell$ to get the estimate of the proposition.  
\end{proof} 

%The proof of the main theorem for Kottwitz varieties is now not more than a formality: 

\begin{proof}[Proof of Theorem~\ref{maintheorem}]
We have the composition $\psi \colon B \hookrightarrow S_\p \surjects \pi_0(S)_{\p}$. The Hecke operator $f$ acts via $\Psi(f)$ on $\pi_0(S)_{\p} = Y$. The map $\psi$ induces an inclusion of the free complex vector space on $Y(\lfq)$ into $A$: $i \colon A_{\textup{Ab}} \hookrightarrow A$. The composition $\Avg \circ i$ is the identity on $A_{\textup{Ab}}$. Thus $\Avg$ induces a decomposition $A = A_{\textup{Ab}} \oplus \Ker(\Avg)$. We claim that $\Ker(\Avg) = A_0$, the subspace generated by the infinite dimensional automorphic representations. To see this, note that $A_{\textup{Ab}}$ is the sum of precisely all the one-dimensional automorphic representations of $G(\A)$ occurring in the character formula of $A$ (Equation~\eqref{previoustheorem}). The remaining (infinite dimensional) representations thus account for $A/A_{\textup{Ab}} = \Ker(\Avg)$. Thus $\Ker(\Avg) = A_0$. The composition $B \subset S_\p \surjects \pi_0(S)_{\p}$ is equivariant for the Hecke operators. Thus we have the formula $f \cdot i(v) = i(\Psi(f) v)$ for all $v \in A_{\textup{Ab}}$ and all Hecke operators $f$ on $G(\Af)$. The theorem now follows: For $v = v_0 + v_{\textup{Ab}}$, we have $T_m(v) = T_m(v_0) + T_m(v_{\textup{Ab}}) = T_m(v_0) + \Psi(T_m)\Avg(v)$. The sequence $T_m(v_0)$ is estimated by Proposition~\ref{mainprop}. This completes the proof.
\end{proof}

\section{A typical example}\label{sect:Hecke}\label{sect:example}

We give two typical examples of sequences of Hecke operators for which one can prove equidistribution. 

\subsection*{First example} Look at a sequence of powers of a fixed operator: Let $T \in \cH(G(\Af))$ be one Hecke operator satisfying the conditions (H1) and (H2) above, and $N(T)/\deg(T) \leq 1$. Then the sequence of powers $T_m := (T)^{*m}$ for the convolution product $*$ satisfies $\lim_{m \to \infty} \tfrac {N(T_m)}{\deg(T_m)} = 0$. By Theorem~\ref{maintheorem} the sequence $T_m$ thus acts with equidistribution on the basic stratum.

\subsection*{Second example} We define an explicit sequence of Hecke operators $T_{r, m} \in \cH(G(\Af))$, where $m$ ranges over the square free integers, and $r$ ranges over the integers $1 \leq r \leq n-1$. Consider the $\Q$-group $G_+ := \Res_{\cK/\Q} G_{\cK}$ with $G_+(\Q) = \cK^\times \times D^\times$. Let $S$ be a finite set of finite, rational primes, such that:
\begin{enumerate}
\item for all primes $\ell$ not lying in $S$, the group $G_+(\ql)$ is a product of general linear groups over finite, unramified extensions of $\ql$; 
\item $K$ splits into a product $K = K_S K^S$, where $K_S$ is a subgroup of $G(\A_{\ff, S})$ and $K^S$ is a subgroup of $G(\Af^S)$;
\item the prime $p$ lies in $S$;
\item the group $K^S$ is hyperspecial. %obtained by taking the $\Zhat^S$-points of a smooth model $\cG$ of $G_+$ over the ring $\Z[\ell^{-1} | \ell \in S]$.
\end{enumerate}

Let $G^+$ be the $\Q$-group $\cK^\times \times \Gl_n(F)$. Then $G^+$ is an inner form of $G_+$, and we have $G^+(\Ql) \cong G_+(\Ql)$ for all primes $\ell$ not in $S$. The group $G^+$ has an obvious model over $\Z$, and thus we have the hyperspecial subgroup $G^+(\Zhat) \subset G^+(\Af)$. Let $m$ and $r$ be integers, where we have $0 \leq r \leq n$ (no condition on $m$). Then, by definition, the operator $T_{r,m}^+$ is defined to be the characteristic function:
\begin{equation}\label{thefunction}
T_{r, m}^+ := \charr \lhk G^+(\Zhat) \cdot (1) \times \diag(\underset{r}{\underbrace{m, m, \ldots, m}}, 1, 1, \ldots 1) \cdot G^+(\Zhat) \rhk \in \cH(G^+(\Af)),
\end{equation}
where we should clarify the notation. We have $G^+(\Zhat) = \widehat \cO_{\cK}^\times \times \Gl_n(\widehat \cO_F)$, where $\widehat \cO_{\cK}^\times$ is the factor of similitudes. With $(1) \times \diag( \ldots)$, we mean an element of $G^+(\Zhat)$ that has trivial factor of similitudes, and $\diag(\ldots)$ describes a diagonal matrix in the general linear group over $\widehat \cO_F$. 
 
Because the group $G_+(\Af^S)$ is isomorphic to $G^+(\Af^S)$, the operator $T_{r, m}^{+S} = \bigotimes_{\ell \notin S} T_{r,m}^{(\ell)}$ lives also in the algebra $\cH(G_+(\Af^S))$. We have the base change morphism
$$
\textup{BC} \colon \cH(G_+(\Af^S)//G_+(\Zhat^S)) \to \cH(G(\Af^S)//K^S).
$$ 
We define the operator $T_{r, m}^{S}$ to be $\textup{BC}(T_{r,m}^{+S})$, and we define
\begin{equation}\label{theheckeoperator}
T_{r, m} := \one_{K^{S}} \otimes T_{r, m}^{S} \in C_c^\infty(G(\Af)//K). 
\end{equation}
%We define the Hecke algebra $\cT \subset C_c^\infty(G(\Af))$ to be the complex algebra generated by the operators $T_{r, m}$. The operators $T_{r, m}$ commute with each other, and satisfy no other algebraic relation. Thus the algebra $\cT$ is isomorphic to the polynomial ring $\C[T_{r, m} | r, m]$ on a countable, infinite number of variables. The module $A$ is semi-simple as $\cH(G(\Af^p))$-module (thus also as $\cT$-module) because we know from our formula in Equation~\eqref{previoustheorem} that all irreducible subquotients occurring in $A$ occur in the discrete spectrum of $G$.

We now verify that $\lim_{m, r \to \infty} T_{r, m} \to 0$. Because we have an explicit sequence, one can give explicit bounds:

\begin{proposition}\label{bounding_proposition}
Let $\eps > 0$. There exists an integer $M > 0$ such that for all square-free integers $m > M$ all integers $r$ with $1 \leq r \leq n-1$,  
$$
\left| \frac {N(T_{r, m})}{\deg(T_{r,m})} \right| < C m^{\eps - [F:\Q] \frac {r(n-r)}2 }. 
$$
\end{proposition}

Before proving Proposition~\ref{bounding_proposition}, we first prove two lemmas. 

\begin{lemma}\label{lem_computation_A}
There exists a constant $C \in \R_{> 0}$ such that for all square-free integers $m > M$ and all integers $r$ with $1 \leq r \leq n-1$ we have
$$
\left| \frac {N(T_{r, m})}{\deg(T_{r,m})} \right| \leq C {n \choose r}^{c_F(m)} m^{-[F:\Q] \tfrac {r(n-r)}2}. 
$$
\end{lemma}

\begin{notation}
Let $m$ be a positive integer, unramified in $F$. In the lemma above we wrote $c_F(m)$ for the number of $\cO_F$-prime ideals $\lambda$ containing the number $m$.
\end{notation}

\begin{lemma}\label{stirlinglemma}
Let $\eps > 0$. There exists an integer $M > 0$ such that for all integers $m > M$ we have ${n \choose r}^{c_F(m)} \leq m^\eps$.
\end{lemma}

Clearly, Proposition~\ref{bounding_proposition} follows once the two lemmas have been established. Let us now prove the two lemmas. 

\begin{proof}[Proof of Lemma~\ref{lem_computation_A}] Let $\ell$ be a prime divisor of $m$. Because $m$ is square free, the prime $\ell$ divides $m$ precisely once, and the $\ell$-th part of the function $T_{r, m}^{+}$ equals
\begin{equation}\label{truebecsquarefree}
\charr \lhk G^+(\zl) \cdot (1) \times \diag(\ell, \ell, \ldots, \ell, 1, 1, \ldots, 1) \cdot G^+(\zl) \rhk \in \cH(G(\ql)).
\
\end{equation}
The element $(1) \times \diag(\ell, \ell, \ldots, \ell, 1, 1, \ldots, 1) \in G^+(\ql)$ is the evaluation at $\ell$ of a \emph{miniscule} cocharacter $\mu_r \in X_*(G^+)$. The field $F$ is unramified above $\ell$, and therefore $\ell$ is a prime element of the local field $F_\lambda$ for every $F$-place $\lambda$ dividing $\ell$. Because $\mu_r$ is miniscule there is a simple formula for the Satake transform of $T_{r, m}^{+(\ell)}$ (cf.~\cite[Thm.~2.1.3]{MR761308}):
\begin{equation}\label{thesattransform}
\cS(T_{r, m}^{+(\ell)}) = 1 \otimes \bigotimes_{\lambda | \ell} q_\lambda^{\frac {r(n-r)}2} \sum_{1 \leq i_1 < i_2 < \cdots < i_r \leq n} X_{i_1} X_{i_2} \cdots X_{i_r}, 
\end{equation}
in the algebra
\begin{equation}\label{thesakalgebra}
\C[X_*(T_{\ql})] = \C[Z] \otimes \bigotimes_{\lambda | \ell} \C[X_1^{\pm 1}, X_2^{\pm 1}, \ldots, X_n^{\pm 1}],
\end{equation}
where $T_{\ql} \subset G^+_{\ql}$ is the diagonal torus. We specify that the big tensor product in these equations ranges over all the $F$-places $\lambda$ lying above $\ell$, and for such an $F$-place $\lambda$, we write $q_\lambda$ for the cardinality of the residue field at $\lambda$. 

Note that, up to constant (which we ignore), the degree $\deg(T_{r,m,\ell})$ equals $\deg(T_{r,m, \ell}^+)$. The degree $\deg(T_{r,m,\ell}^+)$ is the evaluation of the polynomial $\cS(T_{r,m,\ell}^{+})$ at the Hecke matrix of the trivial representation $\varphi_{\textup{Triv}}$, and is therefore made completely explicit at this point. We may now estimate $|\cS(T_{r, m,\ell}^{+})(\varphi_\textup{Triv})|$. If we evaluate the symmetric polynomial
$$
\sum_{1 \leq i_1 < i_2 < \cdots < i_r \leq n} X_{i_1} X_{i_2} \cdots X_{i_r},
$$
at the Hecke matrix of the trivial representation of $\Gl_n(F_\lambda)$, then the largest appearing monomial is $q_\lambda^{(n-1)/2 + (n - 3)/2 + \ldots + (n - 2r + 1)/2} = q_{\lambda}^{r(n-r)/2}$.
Hence the following lower bound:
\begin{equation}\label{preprefinal}
|\cS(T_{r, m, \ell}^{+})(\varphi_{\textup{Triv}})| \geq C_{\textup{Sat}} \prod_{\lambda | \ell} q_{\lambda}^{\tfrac {r(n-r)}2} = C_{\textup{Sat}} \ell^{[F:\Q]\tfrac {r(n-r)}2},
\end{equation} 
here $C_{\textup{Sat}}$ is a certain constant. We have
\begin{equation}\label{preprefinalB}
|N(T_{r,m,\ell}| \leq |N(T_{r, m, \ell}^+)| = C_{\textup{Sat}}\prod_{\lambda | \ell} {n \choose r} = C_{\textup{Sat}}{n \choose r}^{c_F(\ell)}.
\end{equation}
The lemma follows by combining Equations~\eqref{preprefinal} and~\eqref{preprefinalB}. 
\end{proof}

We prove Lemma~\ref{stirlinglemma} only for $F = \Q$; we leave it to the reader to reduce to this case, or to extend the argument below.

\begin{proof}[Proof of Lemma~\ref{stirlinglemma}] We have $(c_\Q(m))! \leq m$. Write $m = \Gamma(x)$ for some $x \in \R_{\geq 0}$ where $\Gamma$ is the usual Gamma function. Then $c_\Q(m) \leq x$ and from Stirling's formula we get
$$
\frac {c_\Q(m)}{\log(m)} \sim \frac {c_{\Q}(m)} {\log {(\sqrt{2 \pi x} e^{-x} x^x)}} \leq \frac 1 {\log(x) - 1 + \tfrac {\log(\sqrt{2 \pi x} )} x}.
$$
The right hand side converges to $0$ for $x \to \infty$. Thus we may find (for any $\eps>0$) an $M$ such that $\exp(c_\Q(m)) \leq m^\eps$ for all $m > M$. This completes the proof. 
\end{proof}

Applying Theorem~\ref{maintheorem} to the explicit sequence $T_{r, m}$ we get

\begin{corollary}
Let $v \in A$ be an element. Then there exists a constant $C \in \R_{>0}$ such that for any $\eps > 0$ there exist an index $M$, such that for all square-free integers $m > M$ and all $r$ with $1 \leq r \leq n-1$ we have
$$
\left| \frac {T_{r, m}(v)}{\deg(T_{r, m})} - \Psi(T_{r, m})(\Avg(v)) \right| \leq C m^{\eps -[F:\Q] \tfrac {r(n-r)}2}. 
$$
\end{corollary}

\section{Endoscopic unitary Shimura varieties}\label{sect.extension}

We extend Thm.~3.13 from our article~\cite{kret1} to a larger class of endoscopic unitary Shimura varieties satisfying a technical simplifying condition (Hypothesis~\ref{coprimehypothesis}). We prove later that this condition is satisfied for unitary Shimura varieties whose basic stratum is finite. We consider a Shimura variety of PEL-type, of type A, as considered by Kottwitz in his article~\cite{MR1124982}. We assume fixed a PEL-datum consisting of
\begin{enumerate}
\item[\textbf{(A1)}] A simple algebra %\footnote{The notation $Y$ is nonstandard and questionable; we use it because the usual notation $B$ conflicts with our notation for the basic stratum.} 
$T$ over a CM field $F$;
\item[\textbf{(A2)}] A positive involution on the algebra $T$ inducing complex conjugation on $F$; 
\item[\textbf{(A3)}] A Hermitian $T$-module $(V, \langle \cdot, \cdot \rangle)$, where $\langle \cdot, \cdot \rangle$ is symplectic;
\item[\textbf{(A4)}] $h \colon \C \to \End_Y(V)_\R$ is a morphism of $\R$-algebras such that $h(\li z) = h(z)^*$ for all $z \in \C$, and the symmetric bilinear form $V_\R \otimes V_\R \to \R$, $(x, y) \mapsto \langle x, h(i) y \rangle$ is positive definite. 
\end{enumerate}

Let $(G, X)$ be the Shimura datum associated to (\textbf{A1}), (\textbf{A2}), (\textbf{A3}) and the morphism $h^{-1}$. We assume that $F$ is of the form $\cK F^+$, with $\cK/\Q$ a quadratic imaginary extension of $\Q$ and $F^+$ a totally real extension of $\Q$. Then the group $G_\cK$ is isomorphic to a product of (Weil-restriction of scalars of) general linear groups. We let $p$ be a prime of good reduction in the sense of Kottwitz \cite[\S5]{MR1124982} and we assume that $p$ splits in $\cK/\Q$. We write $E$ for the reflex field of the Shimura datum. Furthermore we let $K \subset G(\Af)$ be a compact open subgroup of the form $K = K_p K^p$, with $K_p \subset G(\qp)$ hyperspecial and $K^p \subset G(\Afp)$ sufficiently small so that the PEL-type moduli problem of level $K$ is defined over $\cO_E \otimes \Z_{(p)}$ and the variety $S = S_K$ representing this moduli problem is smooth and quasi-projective. Pick an $E$-prime $\p$ above $p$ and let $B$ be the basic stratum of the variety $S_\p$, where $\fq$ is the residue field of $\cO_E$ at $\p$. We pick an embedding of $\li \Q \to \lqp$ extending the embedding of $E$ into $\lqp$ defined by $\p$. We fix once and for all an embedding of $F$ into $\C$, and $\li \Q$ will always mean the algebraic closure of $\Q$ in $\C$. The field $\lfq$ is the residue field of $\lqp$ and the field $\fq$ is the residue field of $E$ at $\p$. Because we have the embeddings $F^+ \subset \li \Q\subset \lqp$, the Galois group $\Gal(\lqp/\qp)$ acts on the set of infinite $F^+$-places $V(F^+)$ and we may identify any $\wp|p$ with a Galois orbit $V(\wp)$ of infinite places. Let $U \subset G$ be the subgroup of elements with trivial factor of similitudes. Then $U(\R)$ is a product of standard real groups: $U(\R) = \prod_{v \in V(F^+)} U(s_v, n - s_v)$ for certain numbers $s_v$. We assume that $s_v \leq \tfrac 12 n$ so that these numbers are well defined.  For each $F^+$-place $\wp$ above $p$ we define $s_\wp := \sum_{v \in \wp} s_v$ (see also \S\ref{sect.first}).  The additional technical condition is the following:

\begin{hypothesis}\label{coprimehypothesis}\label{hyp.noendoc}
There exists an $F^+$-prime $\wp$ such that the number $s_\wp$ is coprime to $n$.
\end{hypothesis}

\begin{theorem}\label{section7thm}
Assume Hypothesis~\ref{coprimehypothesis} is true for the Shimura datum $(G, X)$. The sum
\begin{equation}\label{thmsumfix}
\sum_{x' \in \Fix^B_{\Phi_\p^\alpha \times f^{\infty p}}(\lfq) } \Tr(\Phi_\p^\alpha \times f^{\infty p}, \iota^*(\lql)_x),
\end{equation}
equals 
\begin{align}\label{thmautoms}
\sum_{\pi, \dim(\pi) = 1, \pi_p = \textup{Unr}} & m(\pi) \zeta_\pi^\alpha P(q^\alpha) \cdot \Tr(f^p, \pi^p) + \ \cr 
&+ (-1)^{(n-1)r} \cdot \sum_{\pi, \pi_p = \textup{St. type}} m(\pi) \zeta_\pi^\alpha P(q^\alpha) \cdot \Tr(f^p, \pi^p), 
\end{align}
where in both sums $\pi$ ranges over the discrete automorphic representations of $G$. The number $r$ equals the number of $F^+$-places $\wp|p$ such that $s_\wp > 0$.
\end{theorem}
\begin{remark}
Let $\ell$ be a prime number different from $p$ and fix an isomorphism $\lql \cong \C$ of abstract fields. Without further mention, we will use this isomorphism to turn the complex valued function $f^{p\infty}$ into a $\lql$-valued function. We used this isomorphism also in Formula's~\eqref{thmsumfix} and~\eqref{thmautoms}. 
\end{remark}
\begin{proof}
Let $\alpha$ be a integer. Consider the function $f = f_\infty f_\alpha f^p$ in the Hecke algebra of $G$, where $f_\infty$ is a Clozel-Delorme function for the trivial complex representation of $G_\C$ and $f^p \in \cH(G(\Af^p))$ is any $K^p$-spherical Hecke operator.  Write $\iota$ for the inclusion $B \hookrightarrow S_\p$. Recall that the article~\cite[p.~376]{MR1124982} gives the result: 
\begin{equation}\label{kottwitzequation}
\sum_{x' \in \Fix_{\Phi_\p^\alpha \times f^{\infty p}}(\lfq) } \Tr(\Phi_\p^\alpha \times f^{\infty p}, \iota^* (\lql)_x) = |\ker^1(\Q, G)|\sum_{(\gamma_0; \gamma, \delta)} c(\gamma_0; \gamma, \delta) O_\gamma(f^{\infty p}) TO_{\delta}(\phi_\alpha),
\end{equation}
where the notations are from [\S19, \textit{loc. cit}]. We restrict this formula to the basic stratum $B$ by considering on the right hand side only \emph{basic} Kottwitz triples. The equation then becomes
\begin{align}\label{kottwitzequation2}
\sum_{x' \in \Fix^B_{\Phi_\p^\alpha \times f^{\infty p}}(\lfq) }& \Tr(\Phi_\p^\alpha \times f^{\infty p}, \iota^* (\lql)_x) = \cr 
& = |\ker^1(\Q, G)|\sum_{(\gamma_0; \gamma, \delta)} c(\gamma_0; \gamma, \delta) O_\gamma(f^{\infty p}) TO_{\delta}(\chi_{\sigma c}^{G(E_{\p, \alpha})} \phi_\alpha),
\end{align}
where 
\begin{itemize}
\item $E_{\p, \alpha}$ is the unramified extension of degree $\alpha$ of $E_\p$ (in $\lqp$); 
\item the function $\chi_{\sigma c}^{G(E_{\p, \alpha})}$ is the characteristic function of the subset of $\sigma$-compact elements in $G(E_{\p, \alpha})$; 
\item $\Fix^B_{\Phi_\p^\alpha \times f^{\infty p}}$ is the fibre product $\Fix_{\Phi_\p^\alpha \times f^{\infty p}} \times_{\Delta} B$;
\item $\Delta$ is the diagonal variety in $S_\p \times S_\p$. 
\end{itemize}
By the stabilization argument of Kottwitz in~\cite[\S7]{MR1044820} the right hand side of Equation~\eqref{kottwitzequation2} simplifies to
\begin{equation}\label{endoformula}
\sum_{H \in \cE} \iota(G, H) \cdot \ST_e^{H*}(\chi_c^{G(\qp)} h), 
\end{equation}
(see \cite[Thm.~7.2]{MR1044820}, and the proof preceding to that theorem) where
\begin{itemize}
\item $\cE$ is the (finite) set of (representatives of equivalence classes) of endoscopic groups $H$ associated to $G$ and unramified at all places where the data $(G, K)$ are unramified.  
\item $\ST_e^{H*}$ is the elliptic part of the geometric side of stable trace formula for $H$; it is a sum of stable integral orbitals on the elliptic $(G, H)$-regular elements in $H(\Q)$ (see first displayed formula of \cite[p.~189]{MR1044820}). 
\item The function $\chi_c^{G(\qp)} h \in \cH(H(\A))$ is defined as follows. When we write the product $\chi_c^{G(\qp)} h$ we actually mean the function $h^p \otimes (\chi_c^{G(\qp)} f_\alpha)^{H(\qp)}$, where $(\chi_c^{G(\qp)} f_\alpha)^{H(\qp)}$ is the endoscopic transfer of the function $\chi_c^{G(\qp)} f_\alpha$ on $G(\qp)$ to $H(\qp)$. In particular the function is only changed at the prime number $p$. Thus outside $p$, the function $h$ is the same as the one considered by Kottwitz: The function $h_\infty$ is defined at the second displayed formula of [\textit{loc. cit}, p. 186]; the function $h^{p \infty}$ is an endoscopic transfer of the function $f^{p\infty}$ to the endoscopic group $H(\Afp)$, see Equation~(7.1) of [p.~178, \textit{loc. cit.}].
\item $\iota(G, H)$ is a certain explicit rational number, see the second displayed formula on [\textit{loc. cit.}, p.~189]. 
\end{itemize} 
 Before we finish the proof of Theorem~\ref{section7thm} we prove some lemmas:

\begin{lemma}\label{truncatedconstanttermsvanish}
Let $P \subset G(\qp)$ be a proper standard parabolic subgroup of $G(\qp)$. Then the truncated constant term $\chi_c^{G(\qp)} f_\alpha^{(P)}$ vanishes.
\end{lemma}
\begin{proof} Let $P$ be a parabolic subgroup of $G(\qp)$. We have $f_\alpha^{(P)} = \one_{q^{-\alpha}} \otimes \bigotimes_{\wp | p} \prod f_{n \alpha s}^{(P_\wp)}$, where $P_\wp$ is the $\wp$-th component of $P$. If $P$ is proper, then $P_\wp$ is proper as well. Pick some $\wp|p$ such that $s_\wp$ is coprime to $n$ (Hypothesis~\ref{coprimehypothesis}). We look at the $\wp$-th component $f_\alpha^\wp$ of the function $f_\alpha \in \cH_0(G(\qp))$ via the isomorphism $\cH_0(G(\qp)) \cong \cH_0(\qp^\times) \otimes \bigotimes_{\wp|p} \cH_0(\Gl_n(F_\wp^+))$. In the notation of~\cite{kret1}, we have $f_\alpha^\wp = f_{n \alpha_v s_v}$ [Prop.~3.3, \textit{loc. cit.}]. These constant terms vanish for the proper parabolic subgroups in case $s$ is coprime to $n$ [Lem.~1.9, \textit{loc. cit.}].
\end{proof}

\begin{lemma}\label{prop.noendoscopy}
For any proper endoscopic group $H$ of $G$ we have $(\chi_c^{G(\qp)} f)^H_\alpha = 0$. 
\end{lemma}
\begin{proof}
The transfer $f \rightsquigarrow f^H$ from the function on $G(\A)$ to the endoscopic group $H(\A)$ factors through the transfer from $G$ to its quasi-split inner form $G^*$. At $p$, the group $G(\qp)$ is quasi split and therefore $G(\qp) = G^*(\qp)$ and we take the transfer from functions on $G(\qp)$ to functions on $G^*(\qp)$ to be trivial. We transfer the function $\chi_c^{G(\qp)} f_\alpha$ on $G^*(\qp)$ to $H(\qp)$. We first consider the function $f_\alpha \in \cH_0(G^*(\qp))$ ($\cH_0$ denotes the spherical Hecke algebra). In~\cite[p. 1668, Sect.~3.4, case 2]{MR2800722}, Sug Woo Shin describes explicitly the transfer for quasi-split similitudes unitary groups. He starts by describing the endoscopic groups, and explains that any $H$ can be identified with a group of the form $G(GU^*(n_1) \times GU^*(n_2))$ with $n = n_1 + n_2$ (there are some conditions on the possible partitions $n = n_1 + n_2$ here, but they are of no importance to us). In particular we assume $H$ is the Levi-component of a maximal standard parabolic subgroup $P_H$ of $G^*$. By the second last displayed formula on [\textit{loc. cit}, p.~1668] the transfer of $f_\alpha$ to a function on $H(\qp)$ is given by $f_\alpha^{(P_H(\qp))} \cdot \chi_{\varpi, u}^+$, where $\chi_{\varpi, u}^+$ is some function which we will not need to specify for our argument. The transfer of a conjugacy class in $H(\Qp)$ to a conjugacy class in $G^*(\qp)$ is the obvious construction (\ie induced from the inclusion $H(\qp) \subset G^*(\qp)$). Consequently the function 
\begin{equation}\label{thetransferredfunction}
\lhk \chi_c^{G(\qp)}|_{H(\qp)} \rhk f_\alpha^{(P_H(\qp))} \chi_{\varpi, u}^+ \in \cH(H(\qp)),
\end{equation}
is a transfer of the function $\chi_c^{G(\qp)} f_\alpha$ to $H(\qp)$. Therefore the transfer vanishes by Lemma \ref{truncatedconstanttermsvanish}.
\end{proof}

A function $h$ is called cuspidal if for every non-elliptic semi-simple conjugacy class $\gamma$ the orbital integral $O_\gamma(h)$ vanishes. 

\begin{lemma}
The truncated function $\chi_c^{G(\qp)} f_\alpha$ is cuspidal.
\end{lemma}
\begin{proof}
Any non-elliptic conjugacy class of $G(\Qp)$ is conjugated to an element of $M$ for some proper standard Levi-subgroup $M$ of $G(\qp)$. Let $P$ be the corresponding standard parabolic subgroup of $G$. Then the orbital integral $O_\gamma(\chi_c^{G(\qp)} f)$ is the product of a certain Jacobian factor with the $M$-orbital integral of $\gamma$ of the function $\chi_c^{G(\qp)} f^{(P)} = 0$ (Lemma~\ref{prop.noendoscopy}). Thus the function is cuspidal. 
\end{proof}

\noindent \textit{Continuation of the proof of Theorem~\ref{section7thm}.} \quad Langlands stabilized in~\cite{MR697567} the elliptic part of the trace formula $\T_e^G$ for any connected reductive group\footnote{Langlands assumes additionally that the derived group of $G$ is simply connected.}:
\begin{equation}\label{Langlands}
\T_e^G(k) = \sum_{H \in \cE} \iota(G, H) \cdot \ST_e^H(k^H). 
\end{equation}
In Langlands's formula we have the following notations:
\begin{itemize}
\item $k$ is an arbitrary Hecke operator in the Hecke algebra $\cH(G(\Af))$ of $G$.
\item $\T_e^G$ is the elliptic part of the invariant trace formula~\cite[\S23, p.~145]{MR2192011}. More precisely
$$
\T_e^G(k) := \sum_{\gamma} \vol\lhk G_\gamma(\Q) A_{G}(\R)^+ \backslash G_\gamma(\A) \rhk \cdot O_\gamma(k),
$$
where the sum ranges over (a set of representatives of) the $G(\Q)$-conjugacy classes of $\Q$-elliptic semi-simple elements in $G(\Q)$, and $G_\gamma$ is the centralizer subgroup of $\gamma$ in $G$. The group $A_G \subset G$ is the split component of the center of $G$, and $A_G(\R)^+$ is the topological neutral component of the set of real points of $A_G$. %We did not normalize any Haar measure; for the correct normalizations we refer to Arthur \cite{MR2192011}. 
\item $\iota(G, H)\in \Q$ is a certain constant not depending on $k$ (but it does depend on $H$ and $G$). The constant is defined in \cite[\S8]{MR757954}; we do not repeat its definition here. 
\item $\ST_e^H$ is the elliptic part of the stable trace formula for the group $H$ \cite[p.~189]{MR2192011}. 
\item $h^H$ is an endoscopic transfer of the function $k$ to the group $H(\Af)$. %See \cite{MR2192011} for an explanation of endoscopic transfer. 
\end{itemize}
It is customary to write sometimes $\T_e$ for $\T_e^G$ when the group $G$ is understood; the same remark applies to $\ST_e^H = \ST_e^G$ and other such notations. 

%We apply Equation~\eqref{Langlands} to our group $G$ and our Hecke operator $h = \chi_c^{G(\qp)} f$. For this Hecke operator the transfer $h^H = (\chi_c^{G(\qp)} f)^H$ vanishes whenever $H$ is a \emph{proper} endoscopic group of $G$ (Prop.~\ref{prop.noendoscopy}). Therefore, Equation~\eqref{Langlands} simplifies to 
%\begin{equation}\label{langlandscor} 
%\T_e^G( (\chi_c^{G(\qp)} f )) = \ST_e^{G^*}((\chi_c^{G(\qp)} f )^{G^*}).
%\end{equation} 
%Apply Equation~\eqref{langlandscor} to the group $G$ and then to the group $G^*$ to find:
%\begin{equation}\label{GtraceGstartrace} 
%\T_e^G(\chi_c^{G(\qp)} f) = \ST_e^{G^*}((\chi_c^{G(\qp)} f)^{G^*}) = \T_e^{G^*}((\chi_c^{G(\qp)} f)^{G^*}).
%\end{equation} 
%Thus the elliptic part of the trace formula of $G$ compares with the elliptic part of the trace formula for the quasi-split inner form $G^*$ of $G$ (for our Hecke operator). 

We have the truncated formula of Kottwitz.
\begin{equation}\label{formule8X}
\Tr( \chi_c^{G(\qp)} f, A) = \sum_{H \in \cE} \iota(G, H) \cdot \ST_e^{H*}(\chi_c^{G(\qp)} h).
\end{equation} 
(recall $A$ is the free complex vector space on the basic stratum of the Shimura variety). The functions $h \in H(\A)$ are described by their stable orbital integrals by Kottwitz in the article~\cite{MR1044820}. In particular the function $h_p$ can be taken to be the endoscopic transfer of the (truncated) Kottwitz function to the endoscopic group $H(\A)$ [\textit{loc. cit.}, p.~180, l.~15]. We have verified above that these endoscopic transfers vanish. Thus we have $\ST_e^H(\chi_c^{G(\qp)} h) = 0$ for all proper endoscopic groups $H$. 

By a result of Morel (and Clozel) we have $\ST_e^{H*} = \ST_e^H$ (use~\cite[Prop.~3.3.4]{MR2567740} and cf. [thm 6.2.1, \textit{loc. cit.}], or use \cite[(2.5)]{clozelpurity}) .
Consequently, Equation~\eqref{formule8X} equals 
\begin{equation}\label{formule8XX}
\Tr( \chi_c^{G(\qp)} f, A) = \sum_{H \in \cE} \iota(G, H) \cdot \ST_e^{H}(\chi_c^{G(\qp)} h).
\end{equation} 
Return to the stabilization of the elliptic part of the trace formula
\begin{equation}\label{formule8Y}
T_e^G(\chi_c^{G(\qp)} f) = \sum_{H \in \cE} \iota(G, H) \cdot \ST_e^H( (\chi_c^{G(\qp)} f)^H ).
\end{equation}
Also in this formula only the principal term\footnote{With the \emph{principal term} we mean the term corresponding to the maximal endoscopic group $G^* \in \cE$ of $G$ (which is the quasi-split inner form of $G$).} remains. Furthermore, the principal term of Formula~\eqref{formule8Y} coincides with the principal term of Formula~\eqref{formule8X}. Thus, the right hand side of Formula~\eqref{formule8X} simplifies to 
\begin{equation}\label{blablabla} 
\Tr(\chi_c^{G(\qp)} f, A) = \T_e^G(\chi_c^{G(\qp)} f). 
\end{equation}

We compare $\T_e^G(\chi_c^{G(\qp)} f)$ with the invariant trace formula for the group $G$. We use the results from Section~7 of~\cite{MR939691}. In this section, Arthur shows that under various conditions on the Hecke function $k$ on $G(\A)$ both the invariant spectral expansion and the invariant geometric expansion of the invariant trace formula $I(k)$ simplify. In Corollary 7.3 of [\textit{loc. cit.}] he shows the following. Assume there is a place $v$ such that $\Tr(k_v, \pi_v) = 0$ whenever the representation $\pi_v$ is a constituent of a properly induced representation $\Ind_{P(\qp)}^{G(\qp)}(\sigma)$, where $\sigma$ is an \emph{unitary} representation of $M(\Q_v)$ (where $P = MN$). Under this assumption the spectral side simplifies: 
\begin{equation} 
\I(k) = \sum_{t \geq 0} \Tr(k, R_{\textup{disc}, t}), 
\end{equation}
(here $R_{\textup{disc, t}}$ is the $t$-part of the discrete spectrum, $t \in \R_{\geq 0}$, cf. [\textit{loc. cit.}]). 

We claim that Arthur's condition is true if we take $v:= p$ and $k := \chi_c^{G(\qp)} f$. The group $G(\qp)$ is a product of general linear groups, therefore parabolic induction of an \emph{unitary} representation from a proper Levi subgroup is irreducible~\cite{MR748505}. The representation $\pi_p$ is a full induction of the form $\Ind_{P(\qp)}^{G(\qp)}(\sigma)$ with $P = MN$ a proper parabolic subgroup of $G$. The truncated constant terms of $f$ vanish: $\chi_c^{G(\qp)} f^{(P(\qp))} = 0$ (for $P \subset G$ proper). By van Dijk's theorem adapted to compact traces~\cite[Prop. 3]{kret1} we get $$\Tr(\chi_c^{G(\qp)} f, \pi_p) = \Tr(\chi_c^{G(\qp)} f^{P(\qp)}, \sigma) = 0,$$ Thus Arthur's condition is verified. We get:
\begin{equation}\label{spectralexpansionsimplifies}
\I(\chi_c^{G(\qp)} f) = \sum_{t \geq 0} \Tr(\chi_c^{G(\qp)} f, R_{\textup{disc}, t}). 
\end{equation}
(We borrowed the argument for~\eqref{spectralexpansionsimplifies} from the proof of Corollary 7.5 in [\textit{loc. cit.}].)

The geometric expansion of $I(f)$ also simplifies, because the function $\chi_c^{G(\qp)} f$ is cuspidal at two places (namely $p$ and $\infty$). By Corollary~7.3 of [\textit{loc. cit.}]:
\begin{equation}\label{simpfinalstep} 
\I(\chi_c^{G(\qp)} f) = \T_e^G(\chi_c^{G(\qp)} f). 
\end{equation} 
Combining the previous considerations (Eq.'s~\eqref{Langlands}--\eqref{simpfinalstep}) we get 
\begin{align*}
\Tr(\chi_c^{G(\qp)} f, A) &= \sum_{H \in \cE} \iota(G, H) \cdot \ST_e^H(\chi_c^{G(\qp)} h) = \sum_{H \in \cE} \iota(G, H) \cdot \ST_e^H( (\chi_c^{G(\qp)} f)^H) \cr
&= \T_e^G( \chi_c^{G(\qp)} f ) = \I(\chi_c^{G(\qp)} f) = \sum_{t \geq 0} \Tr(\chi_c^{G(\qp)} f, R_{\textup{disc}, t}).
\end{align*}
We conclude that 
%\begin{equation}\label{conclusiontraceformulaargument}
$$
\Tr(\chi_c^{G(\qp)} f, A) = \sum_{t \geq 0} \Tr(\chi_c^{G(\qp)} f, R_{\textup{disc}, t}). 
$$
%\end{equation} 
The rest of the  argument for Theorem~\ref{section7thm} is the same as the one we carried out for the Kottwitz varieties in our article~\cite[Thm.~3]{kret1} (from Prop.~11 of [\textit{loc. cit.}] onwards). This completes the proof. 
\end{proof}

\section{Equidistribution for endoscopic varieties}\label{sect.final}
We extend the equidistribution result in Theorem~\ref{maintheorem} to a larger class of Shimura varieties of unitary type, but still assuming the basic stratum is finite.

We continue with the notations of the previous section, thus in particular $(G, X)$ is a PEL-type Shimura datum with $G$ a unitary group of similitudes. We write $B$ for the basic stratum of $S$. The condition that $B$ is finite is a condition on the signatures of the group $G$:

\begin{proposition}\label{finitenesshyp}
The basic stratum $B$ is finite if and only if we have
\begin{itemize}
\item[$\star$] for all $F^+$-primes $\wp$ dividing $p$ we have $s_{\wp} \in \{0, 1\}$.
\end{itemize}
\end{proposition}
\begin{lemma}
The condition $\star$ is true if and only if $P(q^\alpha) = 1$ for all positive integers $\alpha$.
\end{lemma}
\begin{proof}
The last example in Section~1 of~\cite{kret1}. 
\end{proof}
\begin{remark}
It is possible that Proposition~\ref{finitenesshyp} can, knowing that the basic stratum is non-empty, be proved using a simple deformation theory argument.  
\end{remark}
\begin{proof}[Proof of Proposition~\ref{finitenesshyp}]
\noindent ``$\Leftarrow$'' Enlarge the divisibility of $\alpha$ so that the root of unity $\zeta^\alpha$ is $1$ for all automorphic representations contributing to Equation~\eqref{thmsumfix} (which are finite in number). Equation~\eqref{thmautoms} simplifies to
\begin{equation}\label{newprevioustheorem}
\sum_{\pi, \dim(\pi) = 1, \pi_p = \textup{Unr}} m(\pi) \cdot \Tr(f^p, \pi^p) + (-1)^{(n-1)r} \cdot \sum_{\pi, \pi_p = \textup{St. type}} m(\pi) \cdot \Tr(f^p, \pi^p). 
\end{equation}
Take $f^p = \one_{K^p}$. The expression
\begin{equation}\label{fixpoints}
\sum_{x' \in \Fix^B_{\Phi_\p^\alpha \times f^{\infty p}}(\lfq) } \Tr(\Phi_\p^\alpha \times f^{\infty p}, \iota^* (\lql)_x)
\end{equation}
is constant when $\alpha$ is sufficiently divisible. Therefore $B$ is finite.

\noindent ``$\Rightarrow$'' We use results from our preprint~\cite{kretnonempty}. In the proof of the main theorem [\textit{loc. cit. }, Thm.~3.5] of that article we construct a Hecke operator $f \in \cH(G(\Af))$ such that the trace of $f$ against $\sum_{i=0}^\infty (-1)^i \uH^i(B_{\lfq}, \iota^*\cL)$ is, for $\alpha$ sufficiently divisible, equal to a certain non-zero constant times $\Tr(\chi_c^{G(\qp)} f_{\alpha}, \pi_p)$ where $\pi_p$ is an unramified twist of the Steinberg representation. From the explicit computation of the compact trace against the Steinberg representation~\cite[end of \S1]{kret1}, the trace $\Tr(\chi_c^{G(\qp)} f_{\alpha}, \pi_p)$, is, for $\alpha$ sufficiently divisible, a polynomial in $q^\alpha$ of strictly positive degree $d > 0$. Now view $H := \sum_{i=0}^\infty (-1)^i \uH^i(B_{\lfq}, \iota^*\cL)$ as a virtual representation of the absolute Galois group $\Gal(\lfq/\fq)$. Writing this module as a linear combination of characters of the group $\Gal(\lfq/\fq)$, we know that in this expression the character $\Phi \mapsto q^d$ must appear with non-zero multiplicity. Consequently $B$ must be infinite.  
\end{proof}

\begin{hypothesis}
We assume henceforward that the basic stratum $B$ is finite. 
\end{hypothesis}

\begin{definition}\label{complexspaceA}
Let $A$ be the free complex vector space on the set $B(\lfq)$. Then $A$ is a module over the Hecke algebra $\cH(G(\Af))$. 
\end{definition}

Fix an integer $\alpha$, sufficiently divisible so that $B(\lfq) = B(\fqa)$. Then (as usual, using the fixed isomorphism $\lql \cong \C$) the expression in Equation~\eqref{fixpoints} becomes the trace of $f^p$ on the complex vector space $A$ (Definition~\ref{complexspaceA}). Thus Equation~\eqref{newprevioustheorem} equals $\Tr(f^p, A)$. We have re-established Equation~\eqref{previoustheorem}, but now for the larger class of Shimura varieties considered in this section. Repeat the argument for Theorem~\ref{maintheorem} using Equation~\eqref{newprevioustheorem} instead of Equation~\eqref{previoustheorem} to find that the equidistribution result is true with exactly the same rate of convergence as before.

\bibliographystyle{plain}
\bibliography{grotebib}

\end{document}